\documentclass{amsart}
\usepackage{amssymb,amsmath,graphicx}
\newtoks\prt
\numberwithin{equation}{section}

\newtheorem{thm}{Theorem}[section]
\newtheorem{question}[thm]{Question}
\newtheorem{lemma}[thm]{Lemma}
\newtheorem{prop}[thm]{Proposition}

\theoremstyle{definition}

\newtheorem{assumption}[thm]{Assumptions}
\newtheorem{convention}[thm]{Convention}

\def\eqn#1$$#2$${\begin{equation}\label#1#2\end{equation}}

\def\B{\mathcal B}

\def\D{\mathcal D}

\def\N{\mathcal N}

\def\ep{\varepsilon}

\def\en{\mathbb N}
\def\er{\mathbb R}

\def \f {\boldsymbol f}
\def \h {\boldsymbol h}
\def \g {\boldsymbol g}

\def\osc{\operatorname{osc}}

\def \reg {\partial _{\kern1pt\text{reg}}}

\def\di{\,\mbox{\rm d}}

\def\abs#1{\left|#1\right|}

\newcommand{\norm}[1]{\left\|#1\right\|}

\begin{document}

\title{Typical martingale diverges at a typical point}
\author{Ond\v{r}ej F.K. Kalenda and Ji\v{r}\'{\i} Spurn\'y}

\address{Department of Mathematical Analysis \\
Faculty of Mathematics and Physic\\ Charles University\\
Sokolovsk\'{a} 83, 186 \ 75\\Praha 8, Czech Republic}

\email{kalenda@karlin.mff.cuni.cz}
\email{spurny@karlin.mff.cuni.cz}


\thanks{Our research was supported by the grant GA\v{C}R P201/12/0290. The second author was also
supported by The Foundation of Karel Jane\v{c}ek for Science and Research.}

\begin{abstract}
We investigate convergence of martingales adapted to a given filtration of finite $\sigma$-algebras.
To any such filtration we associate a canonical metrizable compact space $K$ such that martingales adapted to
the filtration can be canonically represented on $K$. We further show that (except for trivial cases) typical martingale diverges at a comeager subset of $K$. `Typical martingale' means a martingale from a comeager set in any of the standard spaces of martingales. In particular we show that a typical $L^1$-bounded martingale of norm at most one converges almost surely to zero and has maximal possible oscillation on a comeager set.
\end{abstract}
\maketitle

%%%%%%%%%%%%%%%%%%%%%%%%%%%%%%%%%%%%%%%%%%%%%%%%%%%%

\section{Introduction and preliminaries}

The well-known Doob Martingale Convergence Theorem (see, e.g., \cite[Theorem 1.3.2.8 on p.\ 25]{spaces} or \cite[275G]{FR2})
says that any $L^1$-bounded martingale converges almost surely. In some cases the underlying probability space has also a canonical topological structure. It is the case for example of martingales on the Cantor set. In such a case it is also natural to ask how large the set of convergence of a martingale is in the sense of Baire category.

It is well known that the $\sigma$-ideal of null sets is incomparable with the $\sigma$-ideal of meager sets,
in fact the unit interval can be expressed as a union of two sets, one of them meager and the other one Lebesgue null (see, e.g., \cite[Theorem 1.6]{oxtoby}). This easy fact is a prototype of various `paradoxical decompositions' of certain spaces into two sets belonging to different $\sigma$-ideals. For example, in \cite{PT,MaMa,Mat,LPT} such decompositions are used to illustrate deep problems on differentiability, in \cite{Zaj}
a different behaviour of two $\sigma$-ideals in the space of continuous functions is described, in \cite{DoMy} a similar feature is presented in the group of permutations of the natural numbers.

On the other hand, there are `almost everywhere' type results which hold both in the measure sense and in the category sense. It is the case, for example, for results `up to a $\sigma$-porous set', see, e.g. \cite{Dol,BEH,CZ}.

A different behaviour with respect to measure and with respect to category can be also illustrated by the Strong Law of Large Numbers.
Indeed, consider the Cantor set $C=\{0,1\}^{\en}$ with the standard product probability measure. Then for almost all $x\in C$ we have $\lim\frac1n(x_1+\dots+x_n)=\frac12$, while it is easy to check that the set of those $x\in C$ such that the above limit exists is meager.
So, it seems to us that it is natural to clarify the behaviour of martingales with respect to the Baire category. It turns out that the results
are similar but the proofs are not so easy. This question was investigated by the second author and M.~Zelen\'y in \cite{SpZe} for pointwise bounded martingales on the Cantor set. It is shown there, in particular, that for a comeager set of martingales the set of convergence is meager.

In the present paper we show that results analogous to those of \cite{SpZe} hold in a much more general setting. First, we consider not only pointwise bounded martingales, but also the space of $L^1$-bounded martingales and that of $L^p$-bounded martingales for $p\in(1,\infty]$. The case $p=\infty$ covers the mentioned results of \cite{SpZe}. And secondly, we consider not only martingales on the Cantor set, but general martingales adapted to a given filtration of finite $\sigma$-algebras. In fact, given any such filtration, we construct a canonical compact metrizable space together with a Borel probability measure and canonical filtration such that martingales adapted to the original filtration are `isomorphic' to the martingales adapted to the new one. And in this setting we show that a typical martingale (typical in the sense of category in some of the spaces of martingales)
diverges on a comeager set.

Let us point out Theorem~\ref{main:1p} where we show that a typical martingale of $L^1$-norm at most one converges almost surely to zero and diverges in the strongest possible sense on a comeager set. This results may be considered as an ultimate `paradoxical decomposition' provided by martingales.

All the results are formulated and proved for martingales adapted to a given filtration of finite $\sigma$-algebras since it is the easiest nontrivial case. Similar results can be proved by the same methods in a more general case of filtrations of discrete $\sigma$-algebras (i.e., $\sigma$-algebras generated by some countable partition of the underlying space). In this case it is also easy and canonical to represent martingales as sequences
of continuous functions on a Polish space (not necessarily compact). It is discussed in the last section.

The case of a general filtration is more complicated. Even in this case we can represent martingales as sequences of continuous functions on a
completely metrizable space (separable if the probability is of countable type). It can be done using the standard construction of the measure algebra
of a given probability space equipped with the Fr\'echet-Nikod\'ym metric. So, the question on convergence with respect to the Baire category has also a canonical sense in the general case, but it seems to be more involved and we do not know the answer. We will comment possible generalizations in the last section.

The paper is organized as follows: In the rest of this section we collect basic definitions and recall some well-known
results on martingales. In the second section we define several spaces of martingales and topologies on them (norm topology
and topology of pointwise convergence). In the third section we present the above announced construction of a compact
metrizable space canonically associated to a filtration of finite $\sigma$-algebras. In Section 4 we collect our main results.
Section 5 contains several lemmas, in Section 6 we complete the proofs. The last section contains final remarks on possible generalization
of our results.

\medskip

Let us start with the basic definitions.

Let $(\Omega,\Sigma,P)$ be a probability space. I.e., $\Omega$ is a set, $\Sigma$ a $\sigma$-algebra of subsets of $\Omega$ and $P$ a probability measure defined on the $\sigma$-algebra $\Sigma$. A \emph{filtration} is an increasing sequence $(\Sigma_n)$ of $\sigma$-subalgebras of $\Sigma$. Denote by $\Sigma_\infty$ the $\sigma$-algebra generated by $\bigcup_{n\in\en} \Sigma_n$.

In the sequel we will suppose that the described objects are fixed.

A \emph{martingale adapted to the filtration $(\Sigma_n)$} is a sequence $\f=(f_n)$ of functions with the following two properties.
\begin{itemize}
	\item $f_n\in L^1(\Omega,\Sigma_n,P|_{\Sigma_n})$ for each $n\in\en$.
	\item $\int_E f_n\di P=\int_E f_m\di P$ whenever $n\le m$ and $E\in \Sigma_n$.
\end{itemize}

In the sequel we will write shortly $L^1(\Sigma_n)$ in place of $L^1(\Omega,\Sigma_n,P|_{\Sigma_n})$ (and similarly for other $L^p$ spaces).
It is easy to check from the definitions that $\|f_n\|_{L^1(\Sigma_n)}\le \|f_m\|_{L^1(\Sigma_m)}$ for $n\le m$.
A martingale $(f_n)$ is called \emph{$L^1$-bounded} if $\sup_n \|f_n\|_{L^1(\Sigma_n)}<+\infty$. Let us recall several well-known facts.

\begin{prop}\label{p:zaklad} Let $(f_n)$ be an $L^1$-bounded martingale adapted to the filtration $(\Sigma_n)$. Then the following hold:
\begin{itemize}
	\item For $P$-almost all $\omega\in\Omega$ the limit $\lim\limits_{n\to\infty} f_n(\omega)$ exists and is finite.
	\item The limit function $f_\infty$ belongs to $L^1(\Sigma_\infty)$ and $$\|f_\infty\|_{L^1(\Sigma_\infty)}\le \sup_{n\in\en}\|f_n\|_{L^1(\Sigma_n)}.$$
	\item The following assertions are equivalent:
\begin{itemize}
	\item[(i)] $\|f_\infty\|_{L^1(\Sigma_\infty)}= \sup_{n\in\en}\|f_n\|_{L^1(\Sigma_n)}$.
	\item[(ii)] The sequence $(f_n)$ converges to $f_\infty$ in the norm of $L^1(\Sigma_\infty)$.
	\item[(iii)] The sequence $(f_n)$ is uniformly integrable.
	\item[(iv)] For each $n\in\en$ we have $f_n=E(f_\infty|\Sigma_n)$.
\end{itemize}
\end{itemize}
\end{prop}

Let us first explain some notions and notation used in the proposition.

 A bounded set $A\subset L^1(\Sigma)$ is \emph{uniformly integrable} if it satisfies one of the following equivalent
 conditions:
\begin{itemize}
  \item $\forall \varepsilon>0\; \exists c >0\; \forall f\in A: \int_\Omega (|f|-c)^+\di P<\varepsilon$,
	\item $\forall \varepsilon>0\; \exists \delta>0 \;\forall E\in\Sigma, P(E)<\delta\; \forall f\in A: \int_E|f|\di P<\varepsilon$.
\end{itemize}
The first condition follows \cite[Definition 246A and Remark 246B(d)]{FR2}, the second one is equivalent by \cite[Theorem 246G]{FR2}.

Further, if $f\in L^1(\Sigma)$ and $\Sigma'\subset\Sigma$ a $\sigma$-subalgebra, the symbol $E(f|\Sigma')$ denotes the \emph{conditional expectation} of $f$ with respect to $\Sigma'$, i.e., it is a function $g$ from $L^1(\Sigma')$ satisfying
$\int_E g=\int_E f$ for each $E\in \Sigma'$. Such a $g$ exists and is unique (as an element of $L^1(\Sigma')$), cf. \cite[233D and 242J]{FR2}.

Now let us comment the proof of the above proposition.
The first two assertions form content of the Doob Martingale Convergence Theorem (see, e.g., \cite[Theorem 1.3.2.8 on p. 25]{spaces} or \cite[275G]{FR2}). The third assertion follows from \cite[Theorem 1.3.2.9 on p. 26]{spaces} or from \cite[275H and 275I]{FR2}.

\section{Spaces of martingales}

Let $\f=(f_n)$ be a martingale adapted to the sequence $(\Sigma_n)$ and $1\le p\le\infty$. The martingale $\f$ is called \emph{$L^p$-bounded} if $f_n\in L^p(\Sigma_n)$ for each $n\in\en$ and, moreover, $\sup\|f_n\|_{L^p}<+\infty$.
The space of all $L^p$-bounded martingales will be denoted by $M_p$. If we equip $M_p$ with the norm
$$\|\f\|_p=\sup_{n\in\en}  \|f_n\|_{L^p},$$
it will become a Banach space. We will denote by $M_p^u$ the subspace of $M_p$ formed by uniformly integrable martingales.

This definition follows \cite[Definition 1.3.3 on p. 13]{spaces} with notation from \cite[Section 1]{Tr}.

Notice that $\|f_n\|_{L^p}\le\|f_m\|_{L^p}$ whenever $n\le m$. This well-known fact follows for example from \cite[Remark 2 on p. 11]{spaces}. Alternatively, it follows from the martingale property using the following formulas for the norms:

$$\begin{aligned}
\|f_n\|_{L^1}& = \sup\left\{ \sum_{j=1}^k \left| \int_{A_j} f_n\di P\right| : A_1,\dots,A_k\in\Sigma_n\mbox{ pairwise disjoint}\right\},
\\
\|f_n\|_{L^p}& = \sup \left\{ \left|\int_\Omega f_n g\di P\right| :
g \mbox{ a simple $\Sigma_n$-measurable function with }\|g\|_{L^q}\le 1\right\},
\end{aligned}$$
where the second formula holds for $1<p\le\infty$ and $q$ is the dual exponent to $p$.

If $1<p<\infty$, then $M_p=M_p^u=L^p(\Sigma_\infty)$ by \cite[Theorem 1.3.2.13 on p. 27 and the following remark]{spaces}. More precisely, in this case any $L^p$-bounded martingale is uniformly integrable. Moreover, if $\f=(f_n)$ is such a martingale, then the sequence $(f_n)$ converges to $f_\infty$ in the $L^p$-norm. In particular $\|\f\|_p=\|f_\infty\|_{L^p}$. Conversely, if $f\in L^p(\Sigma_\infty)$, then the sequence $(f_n)$ defined by $f_n=E(f|\Sigma_n)$ is an $L^p$-bounded martingale with $f_\infty=f$. It follows that the equality $M_p=L^p(\Sigma_\infty)$ is a canonical isometric identification.

For $p=1$ the situation is more complicated. Firstly, similarly as for $1<p<\infty$ we have a canonical isometric identification $M_1^u=L^1(\Sigma_\infty)$ (since $\|\f\|_1=\|f_\infty\|_{L^1}$ for any uniformly integrable $L^1$-bounded martingale).
However, in general $M_1^u\subsetneqq M_1$. There are examples in the literature and we will see some examples below (see, e.g., the proof of Lemma~\ref{l:m1s} below or Proposition~\ref{p:residual} below). Let us denote by $M_1^s$ the subspace of $M_1$ formed by martingales converging almost surely to zero. Then $M_1$ is the direct sum of $M_1^u$ and $M_1^s$ and the canonical projection to $M_1^u$ has norm one.
Indeed, let $\f\in M_1$. Set $\f^u=(E(f_\infty|\Sigma_n))$ and $\f^s=(f_n-E(f_\infty|\Sigma_n))$. Then $\f^u\in M_1^u$, $\f^s\in M_1^s$, $\f=\f^u+\f^s$ and $\|\f^u\|_1=\|f_\infty\|_{L_1}\le \|\f\|_1$.

Finally, let us look at the case $p=\infty$. We have again $M_\infty=M_\infty^u$. Indeed, any $L^\infty$-bounded martingale is also, say, $L^2$-bounded, and hence uniformly integrable. Moreover, also in this case we have a canonical isometric identification $M_\infty=L^\infty(\Sigma_\infty)$. Indeed, let $\f\in M_\infty$.
Then by the above $\|f_\infty\|_{L^\infty}\le \|\f\|_\infty$. Moreover, since $\f$ is uniformly integrable  we have $f_n=E(f_\infty|\Sigma_n)$ for each $n\in\en$.
Therefore $\|f_n\|_{L^\infty}\le \|f_\infty\|_{L^\infty}$ by \cite[243J(b)]{FR2}.
It follows that $\|f_\infty\|_{L^\infty}=\|\f\|_\infty$. Together with the fact that for any $f\in L^\infty(\Sigma_\infty)$ the sequence $\f=(E(f|\Sigma_n))$ is an $L^\infty$-bounded martingale with $f_\infty=f$ we get the announced isometric identification.

In addition to the norm, we will consider one more topology on the above defined spaces $M_p$ -- the topology of pointwise convergence. More precisely, we can equip $M_p$ by the product topology inherited from $\prod_{n\in\en} L^p(\Sigma_n)$.
We will consider only martingales of norm at most one, i.e., the spaces
$$M_{p,1}=\{\f\in M_p: \|\f\|_p\le 1\}.$$
Then $M_{p,1}$ is a closed subset of the product $\prod_{n\in\en} L^p(\Sigma_n)$, hence it is a completely metrizable space.

\section{Compact space associated to a filtration of finite $\sigma$-algebras}\label{sec:comp}

Let us suppose that the filtration $(\Sigma_n)$ is formed by finite $\sigma$-algebras. Then for each $n\in\en$ there is $\D_n$, a finite partition of $\Omega$ generating $\Sigma_n$. Since the sequence $(\Sigma_n)$ is increasing, $\D_m$ refines $\D_n$ for $m\ge n$. Therefore there is, given $m\ge n$, a unique mapping $\varphi_{m,n}:\D_m\to\D_n$ such that $D\subset\varphi_{m,n}(D)$ for each $D\in\D_m$.
Then the sequence $(\D_n)$ together with the just defined mappings form an inverse sequence. Let us equip each $\D_n$ with the discrete topology and let $K$ be the inverse limit of this inverse sequence. I.e., $K$ can be represented as
$$K=\left\{(D_n)\in\prod_{n\in\en}\D_n: D_m\subset D_n\mbox{ for }m\ge n\right\}.$$
Then $K$ is a zerodimensional metrizable compact space. Let us denote by $\varphi_n$ the canonical projection of $K$ onto $\D_n$. Further, let us define the mapping $\psi:\Omega\to K$ by
$$\psi(\omega)=(D_n)\mbox{ where }\omega\in D_n\in\D_n\mbox{ for }n\in\en.$$
Then $\psi$ is a $\Sigma_\infty$-measurable mapping, i.e., $\psi^{-1}(A)\in\Sigma_\infty$ for each $A\subset K$ open (hence also for $A\subset K$ Borel). Hence, we can define $\tilde P=\psi(P|_{\Sigma_\infty})$, the image of $P$ under $\psi$. Then $\tilde P$ is a Borel probability on $K$.

In the following proposition we collect some basic properties of the compact space $K$ and the probability $\tilde P$. The first and third assertions follow immediately from definitions, the second one follows from the well-known characterization of Cantor set (see, e.g., \cite[Theorem 7.4]{kechris}).

\begin{prop} \
\begin{itemize}
	\item $K$ has an isolated point if and only if there are $n\in \en$ and $D\in\D_n$ such that $D\in\D_m$ for each $m\ge n$.
	\item If $K$ has no isolated points, then $K$ is homeomorphic to the Cantor set.
	\item The support of $\tilde P$ equals $K$ if and only if $P(D)>0$ for each $D\in\bigcup_n\D_n$.
\end{itemize}
\end{prop}

In the sequel we will always assume some properties of the sequence $(\D_n)$ which we now sum up:

\begin{assumption}\label{a:a} \
\begin{itemize}
	\item $P(D)>0$ for each $D\in\D_n$ and $n\in\en$.
	\item For each $n\in\en$ and $D\in\D_n$ there is $m>n$ such that $D\notin\D_m$.
\end{itemize}
\end{assumption}

Under these assumptions the compact space $K$ has no isolated points and the support of $\tilde P$ equals $K$.
Moreover, martingales adapted to the filtration $(\Sigma_n)$ are in a canonical one-to-one correspondence with sequences $(h_n)$ with the following properties:
\begin{itemize}
	\item $h_n:\D_n\to \er$ is a mapping.
	\item For each $n\in\en$ and each $D\in\D_n$ we have
	$$h_n(D)=\frac1{P(D)}\sum \{ P(D')h_{n+1}(D') : D'\in\D_{n+1}, D'\subset D\}.$$
\end{itemize}

Indeed, if $\f$ is a martingale adapted to the filtration $(\Sigma_n)$, then for each $n\in\en$ the function $f_n$ is constant on each element of $\D_n$. Hence, we can define $h_n(D)$ to be the value of $f_n$ on $D$ for each $D\in\D_n$. Conversely,
having a sequence $(h_n)$ with the above properties, define $f_n$ to be the function with domain $\Omega$ which has value $h_n(D)$ at each point of $D$ for $D\in\D_n$.

Further, martingales adapted to $(\Sigma_n)$ can be canonically identified with certain martingales on $(K,\tilde P)$ adapted to the filtration $(\B_n)$, where
$$\B_n=\{\varphi_n^{-1}(A): A\subset\D_n\}, \qquad n\in\en.$$
Indeed, $\B_n$ is a finite $\sigma$-algebra of clopen subsets of $K$, the sequence $(\B_n)$ is increasing
and its union generates the Borel $\sigma$-algebra of $K$. If $\f=(f_n)$ is a martingale adapted to $(\Sigma_n)$, we define a martingale $\g=(g_n)$ adapted to $(\B_n)$ as follows.
$$g_n((D_k)_{k=1}^\infty)=h_n(D_n),\qquad (D_k)_{k=1}^\infty\in K, n\in\en,$$
where $(h_n)$ is the above defined sequence of mappings. Conversely, any martingale adapted to $(\B_n)$ can be (uniquely)
represented in such a way  (if $(g_n)$ is a martingale adapted to $(\B_n)$, then $(g_n\circ\psi)$ is the corresponding martingale adapted to $(\Sigma_n)$). It is clear that $(g_n)$ is uniformly integrable if and only if $(f_n)$ is uniformly integrable and that $\|\f\|_p=\|\g\|_p$ for $p\in[1,\infty]$.

\begin{convention}\label{c:c} Let $\f=(f_n)$ be a martingale adapted to $(\Sigma_n)$ and let $(h_n)$ be the sequence of mappings defined above and $\g=(g_n)$ be the corresponding martingale adapted to $(\B_n)$. We will identify  $f_n$, $h_n$ and $g_n$.
More specifically:
\begin{itemize}
	\item For $D\in\D_n$ we will use $f_n(D)$ in place of $h_n(D)$.
	\item For $x\in K$ we will use $f_n(x)$ in place of $g_n(x)$.
\end{itemize}
\end{convention}

\section{Main results}

In this section we formulate our main results, their proofs are given in Section~\ref{s:pf} below.
The basic setting of all the results is the following.
$(\Omega,\Sigma,P)$ is a probability space, $(\Sigma_n)$ is a filtration of finite $\sigma$-algebras,
the $\sigma$-algebra $\Sigma_n$ is generated by a finite partition $\D_n$. We suppose that Assumptions~\ref{a:a} are valid.
Let $K$ be the associated compact metrizable space. In the formulations we use Convention~\ref{c:c}.

The first result concerns spaces $M_p$ for $p\in[1,\infty)$ equipped with the norm topology. We recall that for $p\in(1,\infty)$ we have $M_p=M_p^u$. For $p=1$ the result holds for all the three spaces -- $M_1$, $M_1^u$ and $M_1^s$.
The theorem says, in particular, that, for any of these spaces there is a comeager set of martingales which diverge
on a comeager set of $K$. Moreover, the divergence is the strongest possible -- limsup is $+\infty$ and liminf is $-\infty$.
This should be compared to Doob's theorem which says that any such martingales converges almost surely. In case of $M_1^s$ it even converges almost surely to zero.

\begin{thm}\label{main:pn} Let $Y=M_1$, $Y=M_1^u$, $Y=M_1^s$ or $Y=M_p$ for some $p\in(1,\infty)$. Then the set
$$\{(\f,x)\in Y\times K: \limsup f_n(x)=+\infty\mbox{ and }\liminf f_n(x)=-\infty\}$$
is a dense $G_\delta$ subset of $Y\times K$. In particular, for all $\f$ from a dense $G_\delta$ subset of $Y$
$$\limsup f_n(x)=+\infty\mbox{ and }\liminf f_n(x)=-\infty$$
for $x$ from a dense $G_\delta$ subset of $K$.
\end{thm}

The next theorem contains the same result for spaces $M_{p,1}$, $p\in[1,\infty)$, equipped with the pointwise convergence topology. Also in this case we get the strongest possible divergence of a comeager set of martingales on a comeager set of points.

\begin{thm}\label{main:pp} Let $p\in[1,\infty)$ be arbitrary. Then the set
$$\{(\f,x)\in M_{p,1}\times K: \limsup f_n(x)=+\infty\mbox{ and }\liminf f_n(x)=-\infty\}$$
is a dense $G_\delta$ subset of $M_{p,1}\times K$. In particular, for all $\f$ from a dense $G_\delta$ subset of $M_{p,1}$
$$\limsup f_n(x)=+\infty\mbox{ and }\liminf f_n(x)=-\infty$$
for $x$ from a dense $G_\delta$ subset of $K$.
\end{thm}

In case $p=1$ we can say even more:

\begin{thm}\label{main:1p}
There is a comeager subset  $A\subset M_{1,1}$ such that for each $\f\in A$ the following hold.
\begin{itemize}
	\item $\f\in M_1^s$, i.e., $f_n\to 0$ almost surely.
	\item $\{x\in K : \limsup f_n(x)=+\infty\mbox{ and }\liminf f_n(x)=-\infty\}$ is a dense $G_\delta$ subset of $K$.
\end{itemize}
\end{thm}

This theorem follows immediately from Theorem~\ref{main:pp} and Proposition~\ref{p:residual} below. By the quoted proposition, $M_{1,1}\cap M_1^s$ is comeager in $M_{1,1}$. It follows that $M_{1,1}\cap M_1^u$ is meager in $M_{1,1}$. Since $M_{1,1}\cap M_1^u$ is clearly dense in $M_{1,1}$, it is meager in itself, so there is no point in studying typical martingales from  $M_{1,1}\cap M_1^u$. It is also natural to ask about the descriptive quality of these subsets -- it is easy to check that $M_{1,1}\cap M_1^u$ is $F_{\sigma\delta}$ in $M_{1,1}$ (this follows from the characterization in Proposition~\ref{p:zaklad}) but we do not know what is the descriptive quality of $M_{1,1}\cap M_1^s$. If we knew it is $G_\delta$, the proof of
Proposition~\ref{p:residual} would be much simpler. Let us formulate this as a question.

\begin{question} Is $M_{1,1}\cap M_1^s$ a $G_\delta$ subset of $M_{1,1}$? Is it at least Borel?
\end{question}

The case $p=\infty$ is different. An $L_\infty$-bounded martingale is under our assumptions just uniformly bounded.
So, the respective sequence of functions is bounded at each point, hence it cannot have infinite limsup or liminf at any point.
However, for $M_{\infty,1}$ equipped with the pointwise convergence topology we get a result similar to the case $p<\infty$.

\begin{thm}\label{main:ip} The set
$$\{(\f,x)\in M_{\infty,1}\times K: \limsup f_n(x)=1\mbox{ and }\liminf f_n(x)=-1\}$$
is a dense $G_\delta$ subset of $M_{\infty,1}\times K$. In particular, for all $\f$ from a dense $G_\delta$ subset of $M_{\infty,1}$
$$\limsup f_n(x)=1\mbox{ and }\liminf f_n(x)=-1$$
for $x$ from a dense $G_\delta$ subset of $K$.
\end{thm}

In the case of $M_\infty$ equipped with the norm topology, the situation is quite different. Martingales with large oscillation at some point are not even dense. Indeed, if we take any everywhere convergent martingale (for example, constant zero martingale), then martingales in a small neighborhood have controlled oscillation at each point. So, the best we can obtain is the following result dealing just with divergence. Let us also stress that the comeager set from the following theorem is not $G_\delta$ but just $G_{\delta\sigma}$. Therefore it requires a different proof.

\begin{thm}\label{main:in} The set
$$\{(\f,x)\in M_\infty\times K: \liminf f_n(x)<\limsup f_n(x)\}$$
is a comeager $G_{\delta\sigma}$ subset of $M_\infty\times K$. In particular, for all $\f$ from a comeager subset of $M_\infty$
$$\liminf f_n(x)<\limsup f_n(x)$$
for $x$ from a comeager subset of $K$.
\end{thm}

Theorems~\ref{main:ip} and~\ref{main:in} are generalizations of results from \cite{SpZe}. In Section 3 of the quoted paper the authors prove essentially the same results in the special case of martingales on the Cantor set.

%%%%%%%%%%%%%%%%%%%%%%%%%%%%%%%%%%%%%%%%%%%%%%%
\section{Auxiliary results}

In this section we collect several lemmas which will be used to prove the main results. These lemmas are of two types.
First, we establish descriptive quality of the relevant sets by showing they are $G_\delta$. This is easy and essentially
well known, but we include formulations and proofs for the sake of completeness. Further, we give some lemmas which enable us to approximate any martingale by a martingale diverging on a large set. For different types of approximation and different types of divergence we need different approaches.

Also in the results of this section we tacitly use Assumptions~\ref{a:a} and Convention~\ref{c:c}.

Let us start by two lemmas on descriptive quality of certain sets. More precisely, the first lemma states that
certain sets are $G_\delta$ and the second one points out the sequences of functions to which the first one will be applied.

\begin{lemma}
\label{l:lsc}
Let $(f_n)$ be a sequence of continuous functions on a topological space  $X$. Then the following sets are $G_\delta$ in $X$:
\begin{itemize}
	\item[(i)] $\{x\in X\colon \limsup f_n(x)=+\infty\}$,
	\item[(ii)] $\{x\in X\colon \liminf f_n(x)=-\infty\}$,
	\item[(iii)] $\{x\in X\colon \limsup f_n(x)\ge c\}$ for any $c\in\er$,
	\item[(iv)] $\{x\in X\colon \liminf f_n(x)\le c\}$ for any $c\in\er$,
  \item[(v)] $\{x\in X\colon \osc f_n(x)\ge c\}$ for any $c>0$.
\end{itemize}
\end{lemma}

\begin{proof} The assertion (ii) follows from (i) and (iv) follows from (iii) (in both cases applied to the sequence $(-f_n)$).
The proofs of (i), (iii) and (v) follow from suitable descriptions of the respective sets:
\begin{itemize}
	\item[(i)] $\bigcap_{k\in\en}\bigcap_{n\in\en}\bigcup_{m=n}^\infty \{x\in X: f_m(x)>k\}$,
	\item[(iii)] $\bigcap_{k\in\en}\bigcap_{n\in\en}\bigcup_{m=n}^\infty \{x\in X: f_m(x)>c-\frac1k\}$,
	\item[(v)] $\bigcap_{k\in\en}\bigcap_{n\in\en}\bigcup_{p=n}^\infty\bigcup_{q=n}^\infty \{x\in X: |f_p(x)-f_q(x)|>c-\frac1k\}$.
\end{itemize}
\end{proof}

\begin{lemma}
\label{l:lsc-na-MK}
Let $Y$ be any of the spaces $M_p$ or $M_{p,1}$ for any $p\in[1,\infty]$. Let $n\in\en$ be arbitrary.
Then the function
\[
\Phi_n:(\f,(E_i))\mapsto f_n(E_n)
\]
is continuous on $Y\times K$.
\end{lemma}

\begin{proof} We will show that this function is continuous at each point of $Y\times K$.
To do that, fix an arbitrary point $(\f, (E_i))\in Y\times K$. Set
$$U=\{(\g,(D_i))\in Y\times K\colon D_n=E_n\}.$$
Then $U$ is an open neighborhood of $(\f,(E_i))$ and for each $(\g,(D_i))\in U$ we have
$$\Phi_n(\g,(D_i))=g_n(E_n)=\frac1{P(E_n)}\int g_n \chi_{E_n}\di P,$$
so $\Phi_n$ restricted to $U$ is a composition of two continuous functions
$$ (\g,(D_i))\mapsto g_n,\ (\g,(D_i))\in U,\mbox{\qquad and\qquad} h\mapsto \frac1{P(E_n)}\int h \chi_{E_n}\di P,\ h\in L^p(\Sigma_n),$$
so it is continuous.
\end{proof}

Now we continue with some lemmas on approximation.
In the following lemma we construct a special martingale which will be used to resolve the cases $M_1^u$, $M_p$ for $p\in(1,\infty)$ and $M_{p,1}$ for $p\in[1,\infty)$.

\begin{lemma}\label{l:ui} Let $N\in\en$ be given. Then there is a martingale $\f$ adapted to the filtration $(\Sigma_n)$ with the properties:
\begin{itemize}
  \item[(a)] $f_1=\dots=f_N=0$.
	\item[(b)] $\f$ is $L^p$-bounded for each $p\in[1,+\infty)$.
	\item[(c)] $\lim f_n(x)=+\infty$ for $x$ from a dense subset of $K$.
	\item[(d)] $\lim f_n(x)=-\infty$ for $x$ from a dense subset of $K$.	
\end{itemize}
\end{lemma}

\begin{proof}	Let $(k_j,D_j)$, $j\in\en$, be an enumeration of the set $\{(n,D):n\in\en,D\in\D_n\}$. For $j\in\en$ we will construct natural numbers
$n_j$, $p_j$ $q_j$ and sets $A_j$, $B_j$, $C_j$, $B_j^0$, $B_j^1$, $C_j^0$ and $C_j^1$ such that the following conditions are fulfilled.
\begin{itemize}
  \item $n_1\ge N$.
	\item $k_j\le n_j<p_j<q_j<n_{j+1}$.
	\item $A_j\in\D_{n_j}$ and $A_j\subset D_j$.
	\item $B_j\in\D_{p_j}\setminus \D_{p_j+1}$, $B_j\subset A_j$ and $P(B_j)<2^{-j}$.
	\item $C_j\in\D_{q_j}\setminus \D_{q_j+1}$, $C_j\subset A_j$, $C_j\cap B_j=\emptyset$ and $P(C_j)<2^{-j}$.
	\item $B_j^0$ and $B_j^1$ are two different elements of $\D_{p_j+1}$ contained in $B_j$, $P(B_j^0)\ge P(B_j^1)$.
  \item $C_j^0$ and $C_j^1$ are two different elements of $\D_{q_j+1}$ contained in $C_j$, $P(C_j^0)\ge P(C_j^1)$.
\end{itemize}
The construction can be performed by straightforward induction using Assumptions~\ref{a:a}.

Now we are going to construct the required martingale $\f$. The construction will be done by induction.
Let $f_1$ be the constant zero function. Further, suppose that for some $n\in\en$ the function $f_n$ is constructed.
Let us describe $f_{n+1}$. It can be represented as a function defined on $\D_{n+1}$. So, fix $D\in\D_{n+1}$ and let $D'\in\D_n$ be the unique element satisfying $D\subset D'$. We set
$$f_{n+1}(D)=\begin{cases} f_n(D')+1, & n=p_j \mbox{ and }D=B_j^1, \\
f_n(D')-\frac{P(B_j^1)}{P(B_j^0)}, & n=p_j \mbox{ and }D=B_j^0, \\
f_n(D')-1, & n=q_j \mbox{ and }D=C_j^1, \\
f_n(D')+\frac{P(C_j^1)}{P(C_j^0)}, & n=q_j \mbox{ and }D=C_j^0, \\
f_n(D') & \mbox{otherwise}. \end{cases}$$

It is clear that the just defined sequence $\f=(f_n)$ is a martingale adapted to the filtration $(\Sigma_n)$ such that $f_1=\dots=f_N=0$. Moreover, it is $L^p$-bounded for any $p\in[1,+\infty)$. To show this it is enough to check that the sequence $(f_n)$ is Cauchy in $L^p$ for each $p\in[1,+\infty)$. Fix $p\in[1,+\infty)$. Then
$$\|f_{n+1}-f_n\|_{L^p}^p = \begin{cases} \begin{aligned}
 1^p P(B_j^1) +  \left(\frac{P(B_j^1)}{P(B_j^0)}\right)^p P(B_j^0)
  &\le P(B_j^1) + P(B_j^0) \\ &\le P(B_j)\le 2^{-j}, \end{aligned} & n=p_j,\\
\begin{aligned} 1^p P(C_j^1) +  \left(\frac{P(C_j^1)}{P(C_j^0)}\right)^p P(C_j^0)
  &\le P(C_j^1) + P(C_j^0)\\& \le P(C_j)\le 2^{-j}, \end{aligned} & n=q_j, \\
  0 & \mbox{otherwise.}\end{cases}$$
Therefore $\sum_{n=1}^\infty  \|f_{n+1}-f_n\|_{L^p}\le 2\sum_{j=1}^\infty (2^p)^{-j}<\infty$, hence the sequence is Cauchy in $L^p$.

It remains to prove the assertions (c) and (d). We will show (c), the proof of (d) is analogous.
Fix $n\in\en$ and $D\in\D_n$. We will find a point $x=(E_i)\in K$ such that $E_n=D$ and $\lim f_n(x)=+\infty$.
We set $E_n=D$ and define $E_1,\dots,E_{n-1}$ in the unique possible way. Set $m_1=n$ and suppose that $r\in\en$ and $m_r\in\en$ are given such that $E_1,\dots,E_{m_r}$ are already defined.

Let $j_r\in\en$ be the smallest number such that $n_{j_r}\ge m_r$ and $A_{j_r}\subset E_{m_r}$. Such a number does exist,
since $(m_r,E_{m_r})=(k_j,D_j)$ for some $j$. Then $n_j\ge k_j=m_r$ and $A_j\subset D_j=E_{m_r}$. Set $m_{r+1}=p_{j_r+1}$,
$E_{m_{r+1}}=B_{j_r}^1$ and $E_i$ for $m_r<i<m_{r+1}$ let be defined in the unique possible way.

This inductive construction provides $x=(E_i)\in K$ with $E_n=D$. Moreover,
$$f_i(x)= f_n(x)+r-1 \mbox{ for }m_r\le i<m_{r+1}, r\in\en,$$
hence $\lim\limits_i f_i(x)=+\infty$.
\end{proof}

The next lemma will be used to resolve the case of $M_\infty$ and $M_{\infty,1}$. It is inspired by \cite[Lemma 3.3]{SpZe}.

\begin{lemma}\label{l:infty} Let $N\in\en$ and $h\in L^\infty(\Sigma_N)$ be given such that $\|h\|_{L^\infty}<1$. Then there is a martingale $\f$ adapted to the filtration $(\Sigma_n)$ with the properties:
\begin{itemize}
  \item[(a)] $f_N=h$.
	\item[(b)] $\|\f\|_\infty\le 1$.
	\item[(c)] $\lim f_n(x)=1$ for $x$ from a dense subset of $K$.
	\item[(d)] $\lim f_n(x)=-1$ for $x$ from a dense subset of $K$.	
\end{itemize}
\end{lemma}

\begin{proof} The proof is analogous to that of Lemma~\ref{l:ui}.	Let $(k_j,D_j)$, $j\in\en$, be an enumeration of the set $\{(n,D):n\in\en,D\in\D_n\}$. For $j\in\en$ we will construct natural numbers
$n_j$, $p_j$ $q_j$ and sets $A_j$, $B_j$, $C_j$, $B_j^0$, $B_j^1$, $C_j^0$ and $C_j^1$ such that the following conditions are fulfilled.
\begin{itemize}
  \item $n_1\ge N$.
	\item $k_j\le n_j<p_j<q_j<n_{j+1}$.
	\item $A_j\in\D_{n_j}$ and $A_j\subset D_j$.
	\item $B_j\in\D_{p_j}\setminus \D_{p_j+1}$ and $B_j\subset A_j$.
	\item $C_j\in\D_{q_j}\setminus \D_{q_j+1}$, $C_j\subset A_j$ and $C_j\cap B_j=\emptyset$.
	\item $B_j^0$ and $B_j^1$ are two different elements of $\D_{p_j+1}$ contained in $B_j$, $P(B_j^0)\ge P(B_j^1)$.
  \item $C_j^0$ and $C_j^1$ are two different elements of $\D_{q_j+1}$ contained in $C_j$, $P(C_j^0)\ge P(C_j^1)$.
\end{itemize}
The construction can be performed by straightforward induction using Assumptions~\ref{a:a}.

Now we are going to construct the required martingale $\f$. The construction will be done by induction.
Set $f_N=h$ and let $f_1,\dots,f_{N-1}$ be defined in the unique possible way (i.e., $f_i=E(h|\Sigma_i)$ for $1\le i\le N-1$). Further, suppose that for some $n\in\en$, $n\ge N$, the function $f_n$ is constructed such that $\|f_n\|_{L^\infty(\Sigma_n)}<1$.
Let us describe $f_{n+1}$. It can be represented as a function defined on $\D_{n+1}$. So, fix $D\in\D_{n+1}$ and let $D'\in\D_n$ be the unique element satisfying $D\subset D'$. We set $\eta(D')=\frac12\min\{1-f_n(D'),f_n(D')+1\}$ and
$$f_{n+1}(D)=\begin{cases} f_n(D')+\eta(D'), & n=p_j \mbox{ and }D=B_j^1, \\
f_n(D')-\frac{P(B_j^1)}{P(B_j^0)}\eta(D'), & n=p_j \mbox{ and }D=B_j^0, \\
f_n(D')-\eta(D'), & n=q_j \mbox{ and }D=C_j^1, \\
f_n(D')+\frac{P(C_j^1)}{P(C_j^0)}\eta(D'), & n=q_j \mbox{ and }D=C_j^0, \\
f_n(D') & \mbox{otherwise}. \end{cases}$$

It is clear that the just defined sequence $\f=(f_n)$ is a martingale adapted to the filtration $(\Sigma_n)$ such that $f_N=h$
and $\|\f\|_\infty\le 1$.

It remains to prove the assertions (c) and (d). We will show (c), the proof of (d) is analogous.
Fix $n\in\en$ and $D\in\D_n$. We will find a point $x=(E_i)\in K$ such that $E_n=D$ and $\lim f_n(x)=1$.
Without loss of generality we may suppose that $n\ge N$.
We set $E_n=D$ and define $E_1,\dots,E_{n-1}$ in the unique possible way. Set $m_1=n$ and suppose that $r\in\en$ and $m_r\in\en$ are given such that $E_1,\dots,E_{m_r}$ are already defined.

Let $j_r\in\en$ be the smallest number such that $n_{j_r}\ge m_r$ and $A_{j_r}\subset E_{m_r}$. Such a number does exist,
since $(m_r,E_{m_r})=(k_j,D_j)$ for some $j$. Then $n_j\ge k_j=m_r$ and $A_j\subset D_j=E_{m_r}$. Set $m_{r+1}=p_{j_r+1}$,
$E_{m_{r+1}}=B_{j_r}^1$ and $E_i$ for $m_r<i<m_{r+1}$ let be defined in the unique possible way.

This inductive construction provides $x=(E_i)\in K$ with $E_n=D$. Moreover, the sequence $(f_i(x))_{i\ge n}$ is non-decreasing,
so it has a limit. Further, let us show that there is some $r\ge 1$ such that $f_{m_r}(x)\ge 0$. If $f_n(x)\ge 0$, we can take $r=1$. If $f_{m_r}(x)<0$ for some $r$, then
$$f_{m_{r+1}}(x)=f_{m_r}(x)+\frac12(f_{m_r}(x)+1),$$ hence
$$f_{m_{r+1}}(x)+1=\frac32(f_{m_r}(x)+1).$$
Since $(\frac32)^k\to\infty$, there must be some $r$ with $f_{m_r}(x)\ge 0$. Let us fix such an $r$. Then for each $j\in\en$ we have $$1-f_{m_{r+j}}(x)=\frac1{2^j}(1-f_{m_r}(x)),$$
hence $f_{m_{r+j}}(x)\to 1$. Thus also $f_i(x)\to 1$.
\end{proof}

We continue by two lemmas on approximation in $M_1$. The construction is in this case more technical.

\begin{lemma}
\label{l:klicove} Let $\f\in M_1$, $\eta,\omega>0$, $n\in\en$ and $E\in \D_n$ be given. Then there exist $\g\in M_1$, $m\in\en$ and $F\in \D_m$ such that
\begin{itemize}
\item [(a)] $g_1=\dots=g_n=0$,
\item [(b)] $\norm{\g}_1<\eta$,
\item [(c)] $m>n$ and $F\subset E$,
\item [(d)] $(f_m+g_m)(F)>\omega$.
\end{itemize}
\end{lemma}

\begin{proof}
We have $\f=\f^s+\f^u$, where $\f^s\in M_1^s$ and $\f^u\in M_1^u$.
We start the proof by finding a number $\ep_1\in (0,\frac{\eta}4)$ satisfying
\begin{equation}
\label{e:epsilon}
(\omega+2)\ep_1<\frac{\eta}8.
\end{equation}
Using Assumptions~\ref{a:a} we find $m_1>n$ and $E_1\subset E$ in $\D_{m_1}$ such that $P(E_1)<\ep_1$.
Further, using uniform integrability we select $\ep_2\in (0,\ep_1)$ such that
\begin{equation}
\label{e:un-int}
\forall k\in\en\ \forall F\in\D_k\colon \left(P(F)<\ep_2\implies \int_F \abs{f^u_k}\, \di P<\frac{\eta}{16}\right).
\end{equation}
Using again Assumptions~\ref{a:a} we find $m_2>m_1$ and $E_2\subset E_1$ in $\D_{m_2}$ such that $P(E_2)<\ep_2$.
Now we find a sequence $(F_i)\in K$ such that
\[
 F_{m_2}=E_2,\quad
\lim f^s_i(F_i)=0\quad\mbox{and}\quad \lim f^u_i(F_i)\in \er.
\]
This is possible since $(f^s_i)$ converges almost surely to zero and $(f^u_i)$ converges almost surely.
Denote $c=\lim f^u_i(F_i)$.
Let $m_3>m_2$ be so large that
\[
f^s_k(F_k)>-1\quad\mbox{and}\quad {f^u_k(F_k)}\in (c-1, c+1)
\]
for $k\ge m_3$.

Now we distinguish two cases.

\smallskip

\emph{Case 1. $\lim P(F_i)=0$.}

In this case we find $m_4>m_3$ so large that
\[
(|c|+\omega+2)P(F_k)<\frac{\eta}8
\]
for $k\ge m_4$.
We find $m>m_4$  such that $P(F_m)<P(F_{m_4})$ (using Assumptions~\ref{a:a}). Pick $F'\in \D_m\setminus\{F_m\}$ such that $F'\subset F_{m_4}$.
We set $F=F_m$.

Let
\[
\alpha=|c|+\omega+2,\quad \alpha'=-\alpha\frac{P(F)}{P(F')}.
\]
Define a martingale $\h$ by the formula
\[
h_k(D)= \begin{cases} \alpha & k\ge m, D\in\D_k,D\subset F, \\
 \alpha' & k\ge m, D\in\D_k, D\subset F', \\
 0 & \mbox{otherwise}.\end{cases}
\]
Then
\[
\norm{\h}_1= |\alpha'|P(F')+|\alpha|P(F)\le 2|\alpha|P(F_{m_4})= 2(|c|+\omega+2)P(F_{m_4})<\frac{\eta}4.
\]
Further,
\[
(f_m+h_m)(F)=f^s_m(F)+f^u_m(F)+h_m(F)\ge -1+c-1+|c|+\omega+2\ge\omega.
\]

\smallskip

\emph{Case 2. $\lim P(F_i)>0$.}

Denote $\beta=\lim P(F_i)$. Then there exists $m_5>m_3$ such that
 $P(F_i)<2\beta$ for $i\ge m_5$. Then for each $m_5\le i\le j$ it holds
\[
\frac{P(F_i)}{P(F_j)}\le \frac{2\beta}{\beta}\le 2 \mbox{\quad for } m_5\le i\le j.
\]
 Using Assumptions~\ref{a:a} we find an index $m>m_5$ for which $P(F_m)<P(F_{m_5})$. Let $F=F_{m}$. Then $\frac{P(F_{m_5})}{P(F)}\le 2$. Pick $F'\in \D_m\setminus\{F\}$, $F'\subset F_{m_5}$. Define
\[
\alpha=|c|+\omega+2,\quad \alpha'=-\alpha\frac{P(F)}{P(F')}.
\]
Define a martingale $\h$ by the formula
\[
h_k(D)= \begin{cases} \alpha & k\ge m, D\in\D_k,D\subset F, \\
 \alpha' & k\ge m, D\in\D_k, D\subset F', \\
 0 & \mbox{otherwise}.\end{cases}
\]
Then
\[
\begin{split}
\norm{\h}_1&= |\alpha'|P(F')+|\alpha|P(F)\le 2|\alpha|P(F_{m_5})= 2(|c|+\omega+2)P(F_{m_5})\\
&\le 2|c|P(F_{m_5})+2(\omega+2)P(E_1)\\
&\le 2(|f^u_m(F)|+1)P(F)\frac{P(F_{m_5})}{P(F)}+2(\omega+2)\ep_1\\
&\le 4\int_F \abs{f^u_m} \di P+2 P(F_{m_5})+2\cdot\frac{\eta}8\\
&\le 4\cdot\frac{\eta}{16}+2\ep_1+\frac{\eta}4<\eta.
\end{split}
\]
Further,
\[
(f_m+h_m)(F)=f^s_m(F)+f^u_m(F)+h_m(F)\ge -1+c-1+|c|+\omega+2\ge\omega.
\]\

The proof is finished.
\end{proof}

\begin{lemma}
\label{l:indukce}
Let $\f\in M_1$, $\eta>0$, $m\in\en$ and $E\in\D_m$. Then there exist $\g\in M_1$ and $(F_i)\in K$ such that $\norm{\f-\g}_{1}<\eta$, $F_m=E$ and $\limsup g_i(F_i)=\infty$.
\end{lemma}

\begin{proof}
For $n\ge 0$, we inductively construct martingales $\f^n\in M_1$, indices $m_n\in\en$ and sets $E_n\in \D_{m_n}$ such that $\f^0=\f$, $m_0=m$, $E_0=E$ and for every $n\in\en$ the following conditions are satisfied:
\begin{itemize}
\item [(a)] $m_n>m_{n-1}$,
\item [(b)] $E_{n}\subset E_{n-1}\subset E_0$, $E_n\in \D_{m_n}$,
\item [(c)] $f^{n}_j=f^{n-1}_j$ for $j\le m_{n-1}$,
\item [(d)] $f^{n}_{m_{n}}(E_{n})>n$.
\item [(e)] $\norm{\f^n-\f^{n-1}}_1<\frac{\eta}{2^{n}}$.
\end{itemize}

Define $\f^0$, $m_0$ and $E_0$ as required. Assume that $n\in\en$ and that we have constructed the required objects up to $n-1$.
We use Lemma~\ref{l:klicove} for $\f=\f^{n-1}$, $\omega=n$, index $m_{n-1}$ in place of $n$, $E=E_n$ and $\frac{\eta}{2^{n}}$ in place of $\eta$.
We obtain $\f^{n}$, $m_{n}>m_{n-1}$, $E_{n}\subset E_{n-1}$ in $\D_{m_{n}}$ such that the conditions (c)-(e) are fulfilled.
This finishes the inductive construction.

By (e) we have
$$\sum_{n=1}^\infty \norm{\f^n-\f^{n-1}}_1 < \eta<  \infty,$$
so $(\f^n)$ is a Cauchy sequence in $M_1$ and, if we denote by $\g$ its limit, we get $\|\g-\f\|_1< \eta$.
Let $(F_i)\in K$ be the sequence satisfying $F_{m_n}=E_n$, $n\ge 0$. Since for any $n\in\en$ and $n'\ge n$ we have
$f^{n'}_{m_n}=f^{n}_{m_n}$, it follows that $g_{m_n}=f^{n}_{m_n}$. Hence
$$g_{m_{n}}(F_{m_{n}})=f^n_{m_n}(F_{m_n})=f^n_{m_n}(E_n)>n,$$
so $\limsup g_i(F_i)=\infty$.
This concludes the proof.
\end{proof}

The next two lemmas deal with approximation in $M_1^s$. The basic idea of construction is the same as in the previous case of $M_1$, but the procedure is much simpler.

\begin{lemma}
\label{l:m1s} Let $\f\in M_1^s$, $\eta,\omega>0$, $n\in\en$ and $E\in \D_n$ be given. Then there exist $\g\in M_1^s$, $m\in\en$ and $F\in \D_m$ such that
\begin{itemize}
\item [(a)] $g_1=\dots=g_n=0$,
\item [(b)] $\norm{\g}_1<\eta$,
\item [(c)] $m>n$ and $F\subset E$,
\item [(d)] $(f_m+g_m)(F)>\omega$.
\end{itemize}
\end{lemma}

\begin{proof}
We start the proof by finding a number $\ep\in (0,\frac{\eta}4)$ satisfying
\begin{equation*}
%\label{e1:epsilon}
(\omega+1)\ep<\frac{\eta}2.
\end{equation*}
Using Assumptions~\ref{a:a} we find $m_1>n$ and $E_1\subset E$ in $\D_{m_1}$ such that $P(E_1)<\ep$.
Now we find a sequence $(F_i)\in K$ such that
\[
 F_{m_1}=E_1\quad\mbox{and}\quad
\lim f_i(F_i)=0.
\]
This is possible since $\f\in M_1^s$ and hence the sequence $(f_i)$ converges almost surely to zero.
Let $m_2>m_1$ be so large that $f_k(F_k)>-1$ for $k\ge m_2$.
We pick $m>m_2$ such that $P(F_m)<P(F_{m_2})$ (using Assumptions~\ref{a:a}).
Pick $F'\in \D_m\setminus\{F_m\}$ such that $F'\subset F_{m_2}$.
We denote $F=F_m$.

Set
\[
\alpha=\omega+1,\quad \alpha'=-\alpha\frac{P(F)}{P(F')}.
\]
Now we will define a martingale $\h$ as follows: We set $h_k=0$ for $k<m$ and
\[
h_m(D)= \begin{cases} \alpha, & D=F, \\
 \alpha', & D = F', \\
 0 & \mbox{otherwise}.\end{cases}
\]
The functions $h_k$, $k>m$, are defined inductively as follows. Fix $k>m$ and suppose that $h_i$ is defined for $i<k$.
Fix any $D\in \D_{k-1}$. If $D\in\D_k$, we set $h_k(D)=h_{k-1}(D)$. If $D\notin \D_k$, we choose $D'\in \D_k$, $D'\subset D$ such that $P(D')\le\frac12P(D)$.
We set
$$h_k(E)=\begin{cases} \frac{P(D)}{P(D')} h_k(D), & E=D',\\
0, & E\in\D_k,E\subset D,E\ne D'. \end{cases}$$
If we perform this construction for each $D\in\D_{k-1}$, we will have constructed $h_k$ completing thus the induction step.

It is clear that the constructed sequence $\h=(h_k)$ is a martingale. Moreover,
\[
\norm{\h}_1= |\alpha'|P(F')+|\alpha|P(F)\le 2|\alpha|P(F_{m_2})\le 2(\omega+1)\varepsilon<\eta.
\]
Further,
\[
(f_m+h_m)(F)> -1+\omega+1=\omega.
\]
Finally, $\h\in M_1^s$. Indeed, set $A_k=\{\omega\in \Omega:h_k(\omega)\ne 0\}$. Then the sequence $(A_k)_{k\ge m}$ is decreasing and, moreover, for each $k\ge m$ there is $k'>k$ such that $P(A_{k'})\le\frac12 P(A_{k})$.

The proof is finished.
\end{proof}

\begin{lemma}
\label{l:m1sind}
Let $\f\in M_1^s$, $\eta>0$, $m\in\en$ and $E\in\D_m$. Then there exist $\g\in M_1^s$ and $(F_i)\in K$ such that $\norm{\f-\g}_1<\eta$, $F_m=E$ and $\limsup g_i(F_i)=\infty$.
\end{lemma}

\begin{proof}
The proof is exactly the same as the proof of Lemma~\ref{l:indukce}. We only should refer to Lemma~\ref{l:m1s} instead to Lemma~\ref{l:klicove} and the sequence $(\f^n)$ should be constructed in $M_1^s$ rather than in $M_1$.
\end{proof}

The last result of this section says that a typical martingale from $M_{1,1}$ converges almost surely to zero.
This is used to prove Theorem~\ref{main:1p}.

\begin{prop}\label{p:residual} The set $M_{1,1}^s=M_{1,1}\cap M_1^s$ is comeager in $M_{1,1}$.
\end{prop}

\begin{proof} In the proof we will use the Banach-Mazur game. Let us briefly recall its setting. There are two players -- I and II and they play in turns nonempty open subsets of $M_{1,1}$. The player I starts by choosing a nonempty open set $U_1\subset M_{1,1}$. Then player II chooses $V_1$, player I chooses $U_2$ and so on. They should obey the rule that the chosen open set is contained in the previous move of the other player. I.e., they produce a decreasing sequence of nonempty open sets
$$U_1\supset V_1\supset U_2\supset V_2\supset U_3\supset\cdots$$
The player II wins if $\bigcap_n V_n\subset M_{1,1}^s$. To show that $M_{1,1}^s$ is comeager it is enough to describe a winning strategy for the player II.

For $\f\in M_{1,1}$, $n\in\en$ and $\varepsilon>0$ we set
$$A(\f,n,\varepsilon)=\{\g\in M_{1,1}\colon \|g_i-f_i\|_{L^1}<\varepsilon \mbox{ for }i=1,\dots,n\}.$$
These sets are open and form a basis of $M_{1,1}$. Without loss of generality we may suppose that both players use only open sets belonging to the basis, hence they construct a sequence
$$A(\f^1,n_1,\varepsilon_1)\supset A(\g^1,m_1,\delta_1)\supset A(\f^2,n_2,\varepsilon_2)\supset A(\g^2,m_2,\delta_2)\supset\cdots$$

We will describe a strategy for the player II. Suppose that the previous move of the player I is $A(\f^k,n_k,\varepsilon_k)$
(where $k\in\en$ is the number of the move). Let us define a martingale $\g^k$  as follows:
\begin{itemize}
  \item Let $q_k>n_k$ be the smallest integer such that $\D_{q_k}\ne\D_{n_k}$. (Such a number exists due to Assumptions~\ref{a:a}.)
	\item For $1\le i< q_k$ set $g_i^k=f_i^k$.
  \item Choose $g^k_{q_k}\in L^1(\Sigma_{q_k})$ such that $E(g^k_{q_k}|\Sigma_{n_k})=g^k_{n_k}$ and  $\|g^k_{q_k}\|_{L^1}=1$.
  \item If $i\ge q_k$ and $g^k_i$ is already defined, we define $g^k_{i+1}$ as follows.

  Fix any $D\in \D_i$. If $D\in\D_{i+1}$, we set $g^k_{i+1}(D)=g^k_i(D)$. If $D\notin \D_{i+1}$, we choose $D'\in \D_{i+1}$, $D'\subset D$ such that $P(D')\le\frac12P(D)$.
We set
$$g^k_{i+1}(E)=\begin{cases} \frac{P(D)}{P(D')} g^k_i(D), & E=D',\\
0, & E\in\D_{i+1},E\subset D,E\ne D'. \end{cases}$$
If we perform this construction for each $D\in\D_{i+1}$, we will have constructed $g^k_{i+1}$.
 \end{itemize}

In this way we have constructed a martingale $\g^k\in M_{1,1}$. Moreover, similarly as in the proof of Lemma~\ref{l:m1s} we see that the sequence of sets
$$\{\omega\in\Omega: g^k_i(\omega)\ne 0\},\ i\ge n_k,$$
is decreasing and the measure of these sets goes to zero. Fix $m_k\ge q_k$ such that
$$P(\{\omega\in\Omega: g^k_{m_k}(\omega)\ne 0\})<\frac1k.$$
Finally, fix $\delta_k\in(0,\min\{\frac1{k^2},\varepsilon_k\})$. The answer of the player II will be $A(\g^k,m_k,\delta_k)$.

It is clear that we have described a strategy for the player II. Next we will show that this strategy is winning. So, suppose that a run
$$A(\f^1,n_1,\varepsilon_1)\supset A(\g^1,m_1,\delta_1)\supset A(\f^2,n_2,\varepsilon_2)\supset A(\g^2,m_2,\delta_2)\supset\cdots$$
was played according to the strategy. We will continue in three steps.

\smallskip

\emph{Step 1:} Let $k\in \en$, $\h\in A(\g^k,m_k,\delta_k)$ and $j\ge m_k$. Then $$P(\{\omega\in\Omega:|h_j(\omega)|\ge k\delta_k\})<\frac2k.$$

To see this, set $E=\{\omega\in\Omega:g^k_{m_k}(\omega)=0\}$. By the construction we have $P(\Omega\setminus E)<\frac1k$. Moreover,
$$\int_{\Omega\setminus E}|g^k_{m_k}|\di P=\int_{\Omega}|g^k_{m_k}| \di P=\|g^k_{m_k}\|_{L^1}=1,$$
so	
$$\int_{\Omega\setminus E}|h_j|\di P\ge \int_{\Omega\setminus E}|h_{m_k}|\di P
\ge 1-\int_{\Omega\setminus E}|g^k_{m_k}-h_{m_k}|\di P\ge 1-{\delta_k},$$
hence
$$\int_{E}|h_j|\di P\le 1-\int_{\Omega\setminus E}|h_j|\di P\le {\delta_k}.$$
Therefore
$${\delta_k}\ge\int_{E}|h_j|\di P\ge k\delta_k P(\{\omega\in E:|h_j(\omega)|\ge k\delta_k\}).$$
Finally,
$$P(\{\omega\in\Omega:|h_j(\omega)|\ge k\delta_k\})\le P(\{\omega\in E:|h_j(\omega)|\ge k\delta_k\}) + P(\Omega\setminus E)<\frac2k.$$
This completes the first step.

\smallskip

\emph{Step 2:} Let $k\in \en$, $\h\in A(\g^k,m_k,\delta_k)$ and $r>m_k$. Then
$$P(\{\omega\in\Omega:|h_j(\omega)|\ge k\delta_k\mbox{ for some }m_k\le j\le r\})<\frac2k.$$

Indeed, for $m_k\le j\le r$ set
$$E_j=\{\omega\in\Omega:|h_j(\omega)|\ge k\delta_k\}\mbox{ and }D_j=E_j\setminus \bigcup_{m_k\le i<j}E_i.$$
Then $E_j\in\Sigma_j$ and hence $D_j\in\Sigma_j$ as well. Let us define a martingale $\h^\prime$ such that for each $i\in\en$ and $D\in\D_i$ we have:
$$h^\prime_i (D) = \begin{cases} h_j(D') & \mbox{ if } j\le i, m_k\le j\le r \mbox{ and } D'\in \D_j, D\subset D'\subset D_j,\\
h_i(D) & \mbox{otherwise.}\end{cases}$$
This formula well defines $h^\prime_i$ since the sets $D_j$, $m_k\le j\le r$ are pairwise disjoint.
Moreover, then $\h^\prime\in  A(\g^k,m_k,\delta_k)$ and
$$\{\omega\in\Omega:|h_j(\omega)|\ge k\delta_k\mbox{ for some }m_k\le j\le r\}=\{\omega\in\Omega:|h^\prime_r(\omega)|\ge k\delta_k\},$$
so we conclude by Step 1.

\smallskip

\emph{Step 3 (conclusion):} If $\h\in\bigcap_{k\in\en} A(\g^k,m_k,\delta_k)$, then $\h\in M_1^s$.

Fix $\h\in\bigcap_{k\in\en} A(\g^k,m_k,\delta_k)$. Suppose $\h\notin M_1^s$, i.e.,
$$P(\{\omega\in\Omega:h_i(\omega)\not\to0\})>0.$$
Then there is some $k\in\en$ such that
$$P(\{\omega\in\Omega:\limsup|h_i(\omega)|>\frac1k\})>\frac2k.$$
Then
$$P(\{\omega\in\Omega:|h_i(\omega)|>\frac1k\mbox{ for infinitely many }i\in\en\})>\frac2k,$$
therefore there is some $r>m_k$ such that
$$P(\{\omega\in\Omega:|h_i(\omega)|>\frac1k\mbox{ for some }m_k\le i\le r\})>\frac2k.$$
Since $k\delta_k<\frac1k$, this contradicts Step 2 and the proof is completed.
\end{proof}

\section{Proofs of the main results}\label{s:pf}

In this section we prove the main theorems (except for Theorem~\ref{main:1p} which has been already proved). All these theorems contain an `in particular' part which follows in all the four cases from the Kuratowski-Ulam theorem, see, e.g., \cite[Theorem 15.1]{oxtoby}. The theorem can be applied since $K$ has a countable base, as a compact metrizable space.

Further, the proofs of three theorems have a similar pattern. We first prove that the respective set is $G_\delta$ and subsequently we show that it is dense.

{\bf Proof of Theorem~\ref{main:pn}.} Set
$$\begin{aligned}
A&=\{(\f,x)\in Y\times K: \limsup f_n(x)=+\infty\},\\
B&=\{(\f,x)\in Y\times K: \liminf f_n(x)=-\infty\}.
\end{aligned}$$
By Lemma~\ref{l:lsc} and Lemma~\ref{l:lsc-na-MK}, both sets are $G_\delta$. So $A\cap B$ is $G_\delta$ as well. To show that $A\cap B$ is dense, it is enough to show that both sets $A$ and $B$ are dense.

If $Y=M_1$, the set $A$ is dense by Lemma~\ref{l:indukce}.
If $Y=M_1^s$, the set $A$ is dense by Lemma~\ref{l:m1sind}. Since $B=-A$, $B$ is in both cases dense as well.

Next suppose that $Y=M_1^u$  or $Y=M_p$ for some $p\in(1,\infty)$ (recall that $M_p=M_p^u$ for $p\in(1,\infty)$).
Let $\f$ be the martingale constructed in Lemma~\ref{l:ui} (the value of $N$ does not matter here, take for example $N=1$). Then $\f\in Y$ and, moreover, for any $\g\in M_\infty$ we have $\f+\g\in Y$
and
\begin{gather*}
\lim (f_n(x)+g_n(x))=+\infty \mbox{ for $x$ from a dense subset of }K,\\
\lim (f_n(x)+g_n(x))=-\infty \mbox{ for $x$ from a dense subset of }K.
\end{gather*}
Further, $M_\infty$ is dense in $Y$ (in the norm of $Y$) since $L^\infty(\Sigma_\infty)$ is dense in $L^p(\Sigma_\infty)$ for any $p\in[1,\infty)$. Hence $\{\f+\g:\g\in M_\infty\}$ is a dense subset of $Y$. It follows that both sets $A$ and $B$ are dense.
\qed\medskip

{\bf Proof of Theorem~\ref{main:pp}.} Set
$$\begin{aligned}
A&=\{(\f,x)\in M_{p,1}\times K: \limsup f_n(x)=+\infty\},\\
B&=\{(\f,x)\in M_{p,1}\times K: \liminf f_n(x)=-\infty\}.
\end{aligned}$$
By Lemma~\ref{l:lsc} and Lemma~\ref{l:lsc-na-MK} both sets are $G_\delta$. So $A\cap B$ is $G_\delta$ as well. To show that $A\cap B$ is dense, it is enough to show that both sets $A$ and $B$ are dense.

Fix an arbitrary element $(\f^0,x)\in M_{p,1}\times K$. Then $x=(D_i)$ where $D_i\in\D_i$ for $i\in\en$. Take $U$ to be any neighborhood of $(\f^0,x)$. Without loss of generality we may suppose that $U$ is a basic neighborhood, i.e.,
$$U=\{ (\g,(E_i))\in M_{p,1}\times K : E_N=D_N \mbox{ and } \|f^0_i-g_i\|_{L^p}<\varepsilon\mbox{ for }i\le N\},$$
where $\varepsilon>0$ and $N\in\en$ are given. Define a martingale $\f^1$ by
$$f^1_i=\begin{cases}(1-\frac\varepsilon2)f^0_i & i\le N,\\ (1-\frac\varepsilon2)f^0_N & i>N.\end{cases}$$
Further, let $\f$ be the martingale constructed in Lemma~\ref{l:ui} for the given value of $N$. Then $\f\in M_p$, thus there is some $\eta>0$ such that $\|\eta\f\|<\frac\varepsilon2$. Take the martingale $\g=\f^1+\eta\f$. Then $\g\in M_{p,1}$ and $\|g_i-f^0_i\|_{L^p}<\varepsilon$ for $i\le N$. Further, $\lim g_n(y)=+\infty$ for $y$ from a dense subset of $K$.
Thus, there is such $y=(F_i)$ with $F_N=D_N$. Then $(\g,y)\in A\cap U$. In the same way we can find $z\in K$ such that
$(\g,z)\in B\cap U$. This shows that $A$ and $B$ are dense.
\qed\medskip

{\bf Proof of Theorem~\ref{main:ip}.} Set
$$\begin{aligned}
A&=\{(\f,x)\in M_{\infty,1}\times K: \limsup f_n(x)=1\},\\
B&=\{(\f,x)\in M_{\infty,1}\times K: \liminf f_n(x)=-1\}.
\end{aligned}$$
Since the values $f_n(x)$ belong to the interval $[-1,1]$, it follows from Lemma~\ref{l:lsc} and Lemma~\ref{l:lsc-na-MK} that both sets are $G_\delta$. So $A\cap B$ is $G_\delta$ as well. To show that $A\cap B$ is dense, it is enough to show that both sets $A$ and $B$ are dense.

Fix an arbitrary element $(\f,x)\in M_{\infty,1}\times K$. Then $x=(D_i)$ where $D_i\in\D_i$ for $i\in\en$. Take $U$ any neighborhood of $(\f,x)$. Without loss of generality we may suppose that $U$ is a basic neighborhood, i.e.,
$$U=\{ (\g,(E_i))\in M_{\infty,1}\times K : E_N=D_N \mbox{ and } \|f_i-g_i\|_{L^\infty}<\varepsilon\mbox{ for }i\le N\},$$
where $\varepsilon>0$ and $N\in\en$ are given.
By Lemma~\ref{l:infty} there is a martingale $\g$ with the properties:
\begin{itemize}
	\item $g_N=(1-\frac\varepsilon2)f_N$,
	\item $g_N\in M_{\infty,1}$,
	\item $\lim g_n(y)=1$ for $y$ from a dense subset of $K$.
	\item $\lim g_n(y)=-1$ for $y$ from a dense subset of $K$.
\end{itemize}

So, we can take $y_1=(E_i)$ such that $E_N=D_N$ and $\lim g_n(y_1)=1$ and $y_2=(F_i)$ such that $F_N=D_N$ and $\lim g_n(y_2)=-1$. Then $(\g,y_1)\in A\cap U$ and $(\g,y_2)\in B\cap U$. This completes the proof that $A$ and $B$ are dense.
\qed\medskip

The proof of the last theorem is different since the set in question is not $G_\delta$ and it is not enough to
prove that it is just dense.

{\bf Proof of Theorem~\ref{main:in}.} The set in question is $G_{\delta\sigma}$ since it equals to $\bigcup_{n\in\en}G_n$ where
$$G_n=\left\{(\f,x)\in M_\infty\times K: \osc (f_k(x))\ge\frac1n\right\}$$
and these sets are $G_\delta$ by Lemma~\ref{l:lsc} and Lemma~\ref{l:lsc-na-MK}. To show that our set is comeager, it is enough to prove that
$$\begin{aligned}\forall\; U\subset M_\infty\times K \mbox{ nonempty open } & \exists\; V\subset U\mbox{ nonempty open } \\ &\exists\; n\in\en: G_n\cap V\mbox{ is dense in }V.\end{aligned}$$
So, fix a nonempty open set $U\subset M_\infty\times K$. Without loss of generality, it is a basic open set, so there are $\f^0\in M_\infty$, $\varepsilon>0$, $N\in \en$ and $D\in\D_n$ such that
$$U=\{(\g,(E_i))\in M_\infty\times K: \sup_i \|g_i-f_i^0\|_{L^\infty}<\varepsilon \mbox{ and }E_N=D \}.$$
Let us distinguish two possibilities:

\smallskip

\emph{Case 1.} There exists $n\in \en$ such that the set
$$\left\{x=(E_i)\in K: E_N=D \mbox{ and } \osc(f^0_k(x))\ge\frac1n\right\}$$
is somewhere dense in $K$. Then there exist $N_1\in\en$, $N_1\ge N$ and $D'\in \D_{N_1}$ with $D'\subset D$ such that for $x$ from a dense subset of
the set $\{(E_i)\in K: E_{N_1}=D'\}$ we have $\osc(f_k(x))\ge \frac1n$. Fix $\delta\in(0,\varepsilon)$ such that
$\delta<\frac1{3n}$. Set
$$V=\{(\g,(E_i))\in M_\infty\times K : \sup_i \|g_i-f_i^0\|_{L^\infty}<\delta \mbox{ and }E_{N_1}=D'\}.$$
Then $V$ is a nonempty open subset of $U$ such that $G_{3n}\cap V$ is dense in $V$.

\smallskip

\emph{Case 2.} For each $n\in \en$ the set
$$\left\{x=(E_i)\in K: E_N=D \mbox{ and } \osc(f^0_k(x))\ge\frac1n\right\}$$
is nowhere dense in $K$.
Fix $n\in\en$ such that $\frac2n<\varepsilon$. Then there are $N_1\in\en$, $N_1\ge N$ and $D'\in\D_{N_1}$  with $D'\subset D$ such that for each $x=(E_i)\in K$ such that $E_{N_1}=D'$ we have $\osc(f^0_k(x))<\frac1n$. Let $\f$ be the martingale given by Lemma~\ref{l:infty}
for $N_1$ in place of $N$ and for $h=0$. Set $\f^1=\f^0+\frac2n\f$. Then $\|\f^1-\f^0\|_\infty\le \frac2n<\varepsilon$ and for $x$ from a dense subset of the set $\{(E_i)\in K: E_{N_1}=D'\}$ we have $\osc(f^1(k))\ge \frac3n$. Set
$$V=\left\{(\g,(E_i))\in M_\infty\times K : \sup_i \|g_i-f_i^1\|_{L^\infty}<\frac1{n} \mbox{ and }E_{N_1}=D'\right\}.$$
Then $V$ is a nonempty open subset of $U$ such that $G_{n}\cap V$ is dense in $V$.

\smallskip

This completes the proof.
\qed\medskip

\section{Remarks on the general case}

In this paper we focused on martingales adapted to a given filtrations of finite $\sigma$-algebras. It it the easiest nontrivial case,
but it is, of course, only a very special case of martingales (as pointed out by one of the referees). In this section we discuss possible generalizations.

\medskip

Let us first focus on general discrete martingales. It means that we have a sequence $(\D_n)$ of countable partitions of $\Omega$ such that for each $n\in\en$ the partition $\D_{n+1}$ refines $\D_n$  and $\Sigma_n=\sigma(\D_n)$. Then we can proceed similarly as in Section 3 to define maps $\varphi_{m,n}$, the space $K$, the maps $\varphi_n$, the map $\psi$ and the measure $\tilde P$. The difference is that $K$ need not be compact, but it is a zero-dimensional Polish space. Further, there is no unique topological description of $K$ in case it has no isolated points -- it may be homeomorphic to the Cantor space $\{0,1\}^\en$, to the Baire space $\en^\en$ or to some other space. We will also suppose that Assumptions~\ref{a:a} are fulfilled (their form is the same). Then $K$ has no isolated points and the support of $\tilde P$ is $K$. Further, Convention~\ref{c:c} can be used in this case as well.

The results in Section 5 can be proved exactly in the same way as in the case the partitions $\D_n$ are finite. Only at two places one should be more careful. It is at the end of the proof of Lemma~\ref{l:m1s} where it is proved that $P(A_k)\to 0$ and at the analogous place of the proof of Proposition~\ref{p:residual}. In this case it is easy to verify, that for any $k\ge m$ there is $k'>k$ such that, say, $P(A_{k'})\le\frac34 P(A_k)$ which is enough to conclude.
Since $K$ has a countable base, the Kuratowski-Ulam theorem works in this case, too. It follows that all the main results
are valid in this more general case as well.

\medskip

The situation for general filtrations is more complicated. However, the analogous questions are also canonical and can be formulated.
Let us describe the situation.

Suppose that $(\Omega,\Sigma,P)$ is a probability space. Let $\N$ denote the $\sigma$-ideal of $P$-null sets. Let $\B$ denote the measure algebra of this probability space (see \cite[321H]{FR3}), i.e., it is the quotient Boolean algebra $\Sigma/\N$ equipped with the probability $\overline{P}$ defined by $\overline{P}([A])=P(A)$ for $A\in\Sigma$ (by $[A]$ we denote the equivalence class of $A$). Let us equip $\B$ with the Fr\'echet-Nikod\'ym metric $\rho$ defined by $\rho([A],[B])=P(A\triangle B)$ (see, e.g., \cite[323A(c)]{FR3}). It is well known and easy to see that this metric is complete on $\B$ (see, e.g., \cite[323G(c)]{FR3}).

Let $f\in L^1(\Sigma)$. Then we can define a function $\overline f:\B\setminus\{0\}\to\er$ by the formula
	$$\overline{f}([A])=\frac1{P(A)}\int_A f\di P.$$
	This function is continuous with respect to the metric $\rho$ (as a ratio of two continuous functions $[A]\mapsto \int_A f\di P$ and $[A]\mapsto P(A)$). Moreover, it satisfies the equality
	$$\overline{P}\left(\bigvee_n[A_n]\right)\overline{f}\left(\bigvee_n[A_n]\right)= \sum_n\overline{P}([A_n])\overline{f}([A_n])$$
	whenever $([A_n])$ is a disjoint sequence (finite of infinite) in $\B\setminus\{0\}$ (the symbol $\bigvee$ denotes the `join' operation in $\B$).

Suppose that we have a filtration $(\Sigma_n)$  of $\sigma$-subalgebras of $\Sigma$. Denote by $\Sigma_\infty$ the $\sigma$-algebra generated by
$\bigcup_{n=1}^\infty\Sigma_n$. Let $\B_n$ be the measure algebra of the probability space $(\Omega,\Sigma_n,P|_{\Sigma_n})$ (for any $n\in\en\cup\{\infty\}$). It is clear that
$$\B_1\subset\B_2\subset\dots\subset\B_\infty\subset\B$$
is an increasing sequence of Boolean subalgebras of $\B$, for each $n\in\en\cup\{\infty\}$ we have $\overline{P|_{\Sigma_n}}=\overline{P}|_{\B_n}$ and the Fr\'echet-Nikod\'ym metric on $\B_n$ is the restriction of $\rho$.

Now we are ready to define the space $K$:
$$ K=\left\{([B_n])\in \prod_{n\in\en}(\B_n\setminus\{0\}): (\forall n\in\en)([B_{n+1}]\le [B_n]) \right\}.$$
Since $\B_n$ is complete in the Fr\'echet-Nikod\'ym metric, $\B_n\setminus\{0\}$ is completely metrizable and the countable product
is again completely metrizable. Moreover, the set $K$ is closed in the product. Indeed, suppose that the sequence $([B_n])$ does not belong to $K$.
It means that there is some $k$ with $[B_{k+1}]\not\le[B_k]$, i.e., $\delta=P(B_{k+1}\setminus B_k)>0$. Suppose that $([C_n])$ is any sequence in the product satisfying $P(C_{k+1}\triangle B_{k+1})<\frac\delta3$ and $P(C_k\triangle B_k)<\frac\delta3$. Then $P(C_{k+1}\setminus C_k)>\frac\delta3$, so $([C_n])\notin K$. It completes the proof that the complement of $K$ is open, hence $K$ is closed. It follows that $K$ is completely metrizable.

Let $\f=(f_n)$ be a martingale adapted to the filtration $(\Sigma_n)$. For any $n\in\en$ let $\overline{f_n}$ be the continuous function on $\B_n\setminus\{0\}$ associated to $f_n$ as above. The martingale condition then means that $\overline{f_n}|_{\B_m\setminus\{0\}}=\overline{f_m}$ for $m\le n$. Therefore such a martingale can be represented as a sequence $(g_n)$ of continuous functions on $K$ defined by
$$g_n(([B_k])_{k=1}^\infty)=\overline{f_n}([B_n]) = \frac1{P(B_n)}\int_{B_n} f_n\di P, \quad ([B_k])_{k=1}^\infty\in K, n\in\en.$$
It follows that it makes sense to investigate pointwise behaviour of such martingales. Therefore it is natural to ask the following question:

\begin{question} Let $(\Omega,\Sigma,P)$ be a probability space and let $(\Sigma_n)$ be a filtration of $\sigma$-subalgebras of $\Sigma$.
Let $K$ be the above defined completely metrizable space. To any martingale $\f=(f_n)$ adapted to the filtration $(\Sigma_n)$ we assign the sequence
$(g_n)$ of continuous functions as above. Similarly as in Convention~\ref{c:c} we will write $f_n(([B_k]))$ to denote $g_n(([B_k]))$.
Further, suppose that
$$\forall n\in\en\;\forall B\in \Sigma_n, P(B)>0\; \exists m\ge n\;\exists C\in\Sigma_m : C\subset B\ \&\ 0<P(C)<P(B).$$
Are in this setting valid the analogues of the results of Section 4?
\end{question}

We conjecture that the analogous results are valid, but the proofs should be more involved. The reason is that constructing ad hoc non-discrete
martingales is not so easy. Further, it is worth to remark that the Kuratowski-Ulam theorem used to prove the `in particular parts' of several
results works only if the space $K$ is separable. It is the case when the probability $P$ is of countable type on $\Sigma_\infty$ (i.e., if it is
of countable type on each $\Sigma_n$). This condition is natural when studying individual martingales. But when we consider all the martingales adapted
to a given filtration, it makes sense to consider also probabilities of uncountable type.

\section*{Acknowledgement}

We are grateful to the referees for their helpful comments which we used to improve the presentation of the paper.

\end{document}